\documentclass[12pt]{amsart}
\usepackage{amsfonts}

\newtheorem{theorem}{Theorem}

\newcommand{\Tr}{{\rm tr}\,}

\newcommand{\cH}{\mathcal{H}}

\begin{document}

\baselineskip 7.5mm

\title{The von Neumann entropy and \\ unitary equivalence of quantum states\footnote{The paper will appear in Linear and Multilinear Algebra.}}

\author{Roman Drnov\v sek}


\begin{abstract}
\baselineskip 7.5mm
K. He, J. Hou, and M. Li have recently given a sufficient and necessary condition for unitary equivalence
of quantum states. This condition is based on the von Neumann entropy. 
In this note we first give a short proof of their result, and then we improve it.
\end {abstract}

\maketitle

\noindent
{\it Key words}: von Neumann entropy, quantum states, unitary equivalence \\
{\it Math. Subj.  Classification (2010)}: 15A15 \\

Throughout the paper, let $\cH$ be a complex Hilbert space with $\textrm{dim} \, \cH = n < \infty$.
A {\it quantum state} on $\cH$ is a positive semidefinite operator with trace equal to $1$. If $I$ is the identity operator on $\cH$,
then the operator $\frac{1}{n} I$ is called {\it the maximal mixed state}.
The {\it von Neumann entropy} $S(\rho)$ of a quantum state $\rho$ is defined by 
$$ S(\rho) = -  \Tr (\rho \log_2 \rho) . $$
If the eigenvalues of $\rho$ are $x_1$, $x_2$, $\ldots$, $x_n$, then 
$$ S(\rho) = - \sum_{i=1}^n x_i \log_2 (x_i) , $$
where by convention we put \ $0 \log_2 0 = 0$. 
The range of the entropy is the interval $[0, \log_2 n]$,
where the maximum value is achieved for the maximal mixed state.
For other properties of the entropy see e.g. \cite{Bh} and the references therein.

Two quantum states $\rho$ and $\sigma$ are said to be {\it unitarily equivalent} if there exists a unitary operator 
$U$ such that $\rho = U \sigma U^*$. In this case $\rho$ and $\sigma$ have the same entropy. However, the converse is not true:
two quantum states with the same entropy are not necessarily unitarily equivalent. 
The von Neumann entropy and the unitary equivalence of quantum states play important roles in the theory of quantum information (see e.g. \cite{NC}). 
In this note we begin with a very short proof of the following sufficient and necessary condition for unitary equivalence
that has been recently given in \cite{HHL}.

\begin{theorem}
\label{infinite_points} 
Let $\rho$ and $\sigma$ be quantum states on $\cH$, and let $a \in (0,1)$. Then 
\begin{equation}
S \left( \lambda \rho + \frac{1-\lambda}{n} \, I \right) = S \left( \lambda \sigma + \frac{1-\lambda}{n} \, I\right)
\textrm{   for all   } \lambda \in (0, a) 
\label{eq1}
\end{equation}
if and only if there exists a unitary operator $U$ such that $\rho = U \sigma U^*$.
\end{theorem}

\begin{proof}
Clearly, the equality (\ref{eq1}) is a necessary condition for the unitary equivalence of $\rho$ and $\sigma$.
To prove its sufficiency, we denote by $f(\lambda)$ and $g(\lambda)$ the left-hand and the right-hand side of (\ref{eq1}), respectively.
Assume that the eigenvalues of $\rho$ and $\sigma$ are $x_1 \geq x_2 \geq \ldots \geq x_n$ and  $y_1 \geq y_2 \geq \ldots \geq y_n$, respectively.
Putting $u_i = x_i - \frac{1}{n}$ for all $i = 1, 2, \ldots, n$, we have $\sum_{i=1}^n u_i = 0$, since $1 = \Tr (\rho) = \sum_{i=1}^n x_i$.
Then, for every $\lambda \in (0, a)$, 
$$ f (\lambda) = - \sum_{i=1}^n  \left( \lambda x_i + \frac{1-\lambda}{n}\right) \, \log_2 \left(  \lambda x_i + \frac{1-\lambda}{n} \right) = $$
$$    = - \sum_{i=1}^n  \left( \lambda u_i + \frac{1}{n} \right) \log_2 \left( \lambda u_i + \frac{1}{n} \right)   $$
and 
$$ \frac{d}{d\lambda} \left(\frac{f (\lambda)}{\lambda} \right) = 
-  \sum_{i=1}^n  \left( - \frac{1}{n \lambda^2} \log_2 \left( \lambda u_i + \frac{1}{n} \right) + 
            \left( u_i + \frac{1}{n \lambda} \right) \cdot \frac{1}{\lambda u_i + \frac{1}{n}} \cdot \frac{u_i}{\ln 2} \right) = $$
$$ = \frac{1}{n \lambda^2} \sum_{i=1}^n  \log_2 \left( \lambda u_i + \frac{1}{n} \right) - \sum_{i=1}^n \frac{u_i}{\lambda \ln 2}                  
= \frac{1}{n \lambda^2} \log_2 \left( \prod_{i=1}^n  \left( \lambda u_i + \frac{1}{n} \right) \right) . $$
Introducing the polynomial $p(\lambda) = \prod_{i=1}^n  \left( \lambda u_i + \frac{1}{n} \right)$ of degree at most $n$, we therefore have 
\begin{equation}
 n \lambda^2 \frac{d}{d\lambda} \left(\frac{f (\lambda)}{\lambda} \right) = \log_2 \left( p(\lambda) \right)
\label{eq2}
\end{equation}
for all $\lambda \in (0, a)$.
Analogously, if we set $v_i = y_i - \frac{1}{n}$ for all $i = 1, 2, \ldots, n$ and we define the polynomial
$q(\lambda) = \prod_{i=1}^n  \left( \lambda v_i + \frac{1}{n} \right)$, then  
\begin{equation}
 n \lambda^2 \frac{d}{d\lambda} \left(\frac{g (\lambda)}{\lambda} \right) = \log_2 \left( q(\lambda) \right) 
\label{eq3}
\end{equation}
for all $\lambda \in (0, a)$.
Now, the equalities (\ref{eq1}),  (\ref{eq2}) and (\ref{eq3}) imply that $p(\lambda) = q(\lambda)$ for all $\lambda \in (0, a)$, and so $p=q$. 
It follows that the $n$-tuples $(u_1, u_2, \ldots, u_n)$ and $(v_1, v_2, \ldots, v_n)$ are the same except for a permutation. 
Therefore, the quantum states $\rho$ and $\sigma$ have the same spectrum, and so they are unitarily equivalent.
\end{proof}

The last part of the proof of Theorem \ref{infinite_points} leads to a question of whether it is enough to verify the condition  
(\ref{eq1}) only for finitely many scalars $\lambda \in [0, 1]$. The following theorem shows that $2n$ scalars are sufficient.

\begin{theorem}
\label{finite_points} 
Let $\rho$ and $\sigma$ be quantum states on $\cH$, and let 
$0 < \lambda_1 < \lambda_2 < \ldots < \lambda_{2n} \le 1$.  
Then 
\begin{equation}
S \left( \lambda_i \rho + \frac{1-\lambda_i}{n} \, I \right) = S \left( \lambda_i \sigma + \frac{1-\lambda_i}{n} \, I\right)
\textrm{   for all   } i = 1, 2, \ldots, 2n  
\label{eq4}
\end{equation}
if and only if there exists a unitary operator $U$ such that $\rho = U \sigma U^*$.
\end{theorem}

\begin{proof}
Keep the notation from the proof of Theorem \ref{infinite_points}. The first derivative of $f$ with respect to $\lambda$ is 
$$ f^{\prime} (\lambda) = - \sum_{i=1}^n  u_i \cdot \log_2 \left( \lambda u_i + \frac{1}{n} \right) - \sum_{i=1}^n \frac{u_i}{\ln 2} 
= - \sum_{i=1}^n  u_i \cdot \log_2 \left( \lambda u_i + \frac{1}{n} \right) , $$
since $\sum_{i=1}^n u_i = 0$. Since the second derivative of $f$ is 
$$ f^{\prime \prime} (\lambda) = - \sum_{i=1}^n  \frac{u_i^2}{\ln 2 \cdot \left( \lambda u_i + \frac{1}{n} \right)} , $$
we obtain that  
$$ f^{\prime \prime} (\lambda) \cdot \prod_{i=1}^n  \left( \lambda u_i + \frac{1}{n} \right)  $$
is a polynomial of degree at most $n-1$. In fact, the coeficient at $\lambda^{n-1}$ is 
$$ \frac{(u_1 + u_2 + \ldots + u_n) u_1 u_2 \ldots u_n}{\ln 2} = 0 , $$
so that the degree is at most $n-2$. Analogously, 
$$ g^{\prime \prime} (\lambda) \cdot \prod_{i=1}^n  \left( \lambda v_i + \frac{1}{n} \right) $$
is a polynomial of degree at most $n-2$. 
Then the function $h = f - g$ has the property that 
$$ h^{\prime \prime} (\lambda) \cdot \prod_{i=1}^n  \left( \lambda u_i + \frac{1}{n} \right) \cdot \prod_{i=1}^n  \left( \lambda v_i + \frac{1}{n} \right)  $$
is a polynomial of degree at most $2n-2$. Assume that the function $h^{\prime \prime}$ is non-zero. 
Then it has at most $2n-2$ zeros. It follows that $h^{\prime}$ has at most $2n-1$ zeros on the interval $[0, 1]$, and so 
the function $h$ has at most $2n$ zeros on the interval $[0, 1]$. However, $h(0) = f(0) - g(0) = 0$ and 
$h(\lambda_i) = 0$ for all $i = 1, 2, \ldots, 2n$ by (\ref{eq4}). This contradiction proves that $h^{\prime \prime} = 0$. 
Since $f^{\prime}(0) = g^{\prime}(0) = 0$, we have $h^{\prime}(0) = 0$, so that $h(\lambda) = 0$ for all $\lambda \in [0, 1]$. 
Now, we apply Theorem \ref{infinite_points} to conclude that $\rho$ and $\sigma$ are unitarily equivalent.
\end{proof}

\vspace{5mm}
{\bf
\begin{center}
 Acknowledgment
\end{center}
} 
The author was supported in part by the Slovenian Research Agency.

\vspace{3mm}

\noindent
Roman Drnov\v sek \\
Department of Mathematics \\
Faculty of Mathematics and Physics \\
University of Ljubljana \\
Jadranska 19 \\
SI-1000 Ljubljana \\
Slovenia \\
e-mail : roman.drnovsek@fmf.uni-lj.si 

\end{document}